\documentclass[12pt]{amsart}
\usepackage{amscd,amsthm,amssymb,amsfonts,amsmath,epsfig,epic,bbm,url, latexsym}
\usepackage[arrow,matrix,tips,2cell]{xy}

\theoremstyle{plain}
\newtheorem{theorem}{Theorem}
\newtheorem{lemma}[theorem]{Lemma}
\newtheorem{proposition}[theorem]{Proposition}

\theoremstyle{definition}
\newtheorem{definition}[theorem]{Definition}

\theoremstyle{remark}
\newtheorem*{remark}{Remark}
\newtheorem*{acknowledgments}{Acknowledgments}

\newcommand{\Z}{\mathbb Z}    % Integers
\newcommand{\R}{\mathbb R}    % Reals
\newcommand{\Aff}{\operatorname{Af{}f}}

\newcommand{\<}{\langle}   %\< is not defined yet.
\renewcommand{\>}{\rangle} %\> is already defined.

\newcommand{\length}{\operatorname{length}}

\newcommand{\ignore}[1]{\relax}

\newcommand{\supp}{\operatorname{supp}}

\newcommand{\Div}{\operatorname{Div}}

\begin{document}

\title[$C \not \sim C^-$: a tropical point of view]
{$C$ is not equivalent to $C^-$ in its Jacobian: a tropical point of view}

\author{Ilia Zharkov}
\address{Kansas State University, 138 Cardwell Hall, Manhattan, KS 66506}
\email{zharkov@math.ksu.edu}

\begin{abstract}
We show that the Abel-Jacobi image of a tropical curve $C$ in its  Jacobian $J(C)$ is not algebraically equivalent to its reflection by using a simple calculation in tropical homology.
\end{abstract}
\maketitle

\footnote{The research is partially supported by the NSF FRG grant DMS-0854989.}

\section{Introduction}
A classical result of Ceresa \cite{ceresa} states that for a generic smooth projective complex curve $C$ of genus $g\ge 3$ the cycles $W_k$ and $W_k^-$, the image of $W_k$ under the involution $(-1): x\to -x$, are not algebraically equivalent in the Jacobian $J(C)$. In this article we will give an argument for analogous tropical statement.

There is a folklore dictionary between complex objects existing in families and their tropical counterparts  appearing as degenerations. For example, if a family of complex curves degenerates to a tropical curve then the corresponding family of Jacobians degenerates in certain sense to the Jacobian of the limiting tropical curve. Moreover, families of cycles in the Jacobians should become tropical cycles (balanced weighted polyhedral complexes with rational slopes) in the tropical Jacobian. Studying tropical cycles, which are essentially linear objects, is a much simpler task than studying the classical cycle problems.

Several correspondences between classical and tropical objects have been established (cf., e.g. \cite{alexeev}, or more recent \cite{viviani} and references therein). 
However the author is unaware of a general assertion that tropical algebraic inequivalence of cycles implies classical inequivalence for generic members of the family. Nor is this the goal of this paper.

Rather our main objective is to introduce a new tropical tool -- the determinantal form. In the case of the Jacobian of a genus 3 curve this is a constant 2-form with values in the integral elements of the cotangent space. If a tropical chain provides an algebraic equivalence between two 1-cycles then the determinantal form by construction vanishes on every 2-cell in the support of this chain. Analogs of the determinantal form exist in higher dimensions and may be used to algebraically distinguish cycles in more complicated tropical varieties. This will potentially give a new (tropical) approach for proving algebraic inequivalences in the classical setting, once the dictionary is established.

\begin{acknowledgments}
The author is very greatful to the referee for pointing out several inaccuracies in the original version of the manuscript.
\end{acknowledgments}

\section{Tropical curves and Jacobians}
\subsection{Tropical curves}
Let $\Gamma$ be a connected finite graph and ${\mathcal V}_1(\Gamma)$ be the set of its 1-valent vertices. We say $\Gamma$ is a metric graph if the topological space $\Gamma\setminus{\mathcal V_1}(\Gamma)$ is given a complete metric structure and $\Gamma$ is its compactification. In particular, all leaves must have infinite lengths. A new two-valent vertex inserted in the interior of any edge produces another metric graph which we set to be {\em equivalent} to the original one. 

\begin{definition}
A {\em tropical curve $C$} is an equivalence class of such metric graphs. Its genus is $g=b_1(\Gamma)$ for any representative $\Gamma$.
\end{definition}

For purposes needed in this paper it is enough to consider finite graphs. That is we can always remove all (infinite) leaves. That affects neither the construction of the Jacobian $J(C)$ nor the Abel-Jacobi map and its image $W_1\subset J(C)$ of $C$, which are the main objects of our interest.

A {\em divisor} $D=\sum a_i p_i$ on the curve $C$ is a formal linear combination of points in $C$ with integral coefficients. We say $D$ has degree $d=\sum a_i$, and $D$ is {\em effective} if all $a_i\ge 0$. Divisors of degree zero form an abelian group.

\subsection{Tropical tori}
Let $V$ be a $g$-dimensional real vector space containing two lattices $\Gamma_1,\Gamma_2$ of maximal rank, that is $V \cong \Gamma_{i}\otimes \R$ for $i=1,2$. Suppose we are given an isomorphism $Q:\Gamma_1\to \Gamma_2^*$, which is symmetric if thought of as a bilinear form on $V$. 

\begin{definition}
The torus $J=V/\Gamma_1$ is the {\em principally polarized tropical torus} with $Q$ being its polarization. The tropical structure on $X$ is given by the lattice $\Gamma_2$. If, in addition $Q$ is positive definite, we say that $J$ is {\em abelian variety}.
\end{definition}

\begin{remark}
The map $Q:\Gamma_1\to \Gamma_2^*$ provides an isomorphism of $J=V/\Gamma_1$ with the tropical torus $V^*/\Gamma_2^*$. The tropical structure on the latter is provided by the lattice $\Gamma_1^*$. In the case of the Jacobian it is very difficult to distinguish between $V/\Gamma_1$ and $V^*/\Gamma_2^*$.
\end{remark}

The data $(V, \Gamma_1, \Gamma_2, Q)$ above is equivalent to a non-degenerate real-valued quadratic form $Q$ on a free abelian group $\Gamma_1\cong \Z^g$. The second lattice $\Gamma_2 \subset V := \Gamma_1\otimes \R$ is then identified with the lattice dual to the image of $\Gamma_1$ under the isomorphism $Q: V\to (V)^*$.

\subsection{Jacobian of a curve}
Let $C$ be a tropical curve of genus $g$. Let $\Gamma_1=H_1(C,\Z) \cong \Z^g$.  We define the symmetric  bilinear form $Q$ on $\Gamma_1$ as $Q(\gamma, \gamma) = \length (\gamma)$ on simple cycles and extend $Q$ bilinearly to arbitrary pairs of cycles. We call the resulting tropical torus $J(C)$ the {\em Jacobian} of the curve $C$.

There is a natural way to visualize the vector space $V$ and the second lattice $\Gamma_2\subset V$ in geometric terms as follows. Let $\Aff$ be the sheaf of $\Z$-affine functions (in some coordinate charts on $C$). Define the integral cotangent local system
${\mathcal T_\Z}^*$ on $C$ by the following exact sequence of sheaves:
\begin{equation*}
0 \longrightarrow \R  \longrightarrow \Aff \longrightarrow {\mathcal T_\Z}^*\longrightarrow
0.
\end{equation*}
The rank $g$ free abelian group of {\em 1-forms} $\Gamma_2^*=\Omega_\Z(C)$ on $C$ is formed by the global sections of ${\mathcal T_\Z}^*$ . Each such form can be thought of as an integral circuit on $C$ satisfying Kirchhofs's law. Then $V=\Omega(C)^*$ is the vector space of $\R$-valued linear functionals on $\Omega_\Z(C)$. The integral cycles in $H_1(C,\Z)$ (i.e. the elements of $\Gamma_1$) become linear functionals on $\Omega(C)$ by integration. The isomorphism $Q:\Gamma_1 \cong \Gamma_2^*$ is the tautological identification between integral cycles and integral circuits on $C$.

Then the definition of the tropical Jacobian
$$J(C):= V/\Gamma_1 = \Omega(C)^*/H_1(C,\Z)$$
is analogous to the classical one for  Riemann surfaces.

\subsection{The Abel-Jacobi map}
Let us fix a reference point $p_0\in C$. Then we can identify the set $\Div^d(C)$ of degree $d$ divisors with $\Div^0(C)$, thus giving a group structure on $\Div^d(C)$ for any $d$.
Given a divisor
$D=\sum a_i p_i$ we choose paths from $p_0$ to $p_i$. Integration along these
paths defines a linear functional on $\Omega_\Z(C)$:
$$\hat\mu(D)(\omega)=\sum a_i \int_{p_0}^{p_i} \omega. $$
For another choice of paths the value of $\hat\mu(D)$ on $\omega$ will differ from the above by $\int_\gamma \omega$ for some $\gamma\in\Gamma_1$. Thus, we get a well-defined {\em
Abel-Jacobi map} $\mu^d:\Div^d(C)\to J(C)$.

We restrict our attention to the effective divisors of degree $d$. As was noted in \cite{MZ2} the Abel-Jacobi map $\mu^1: C\to J(C)$ is a tropical morphism, and so are the maps from the curve's $k$-th powers $\mu^k: C^k\to J(C)$. Let $W_{k} \subset J(C)$ denote its image. 

\section{Algebraic cycles in $J(C)$}
Algebraic $k$-cycles in $J(C)$ are weighed balanced $k$-dimensional polyhedral complexes with $\Gamma_2$-rational slopes. The $W_k$ provide examples of such cycles since they are images of tropical varieties (the powers of $C$) under tropical morphisms (cf. \cite{Mik06}).
 
 \begin{figure}[htb]
    \centering
    \includegraphics[width=4in]{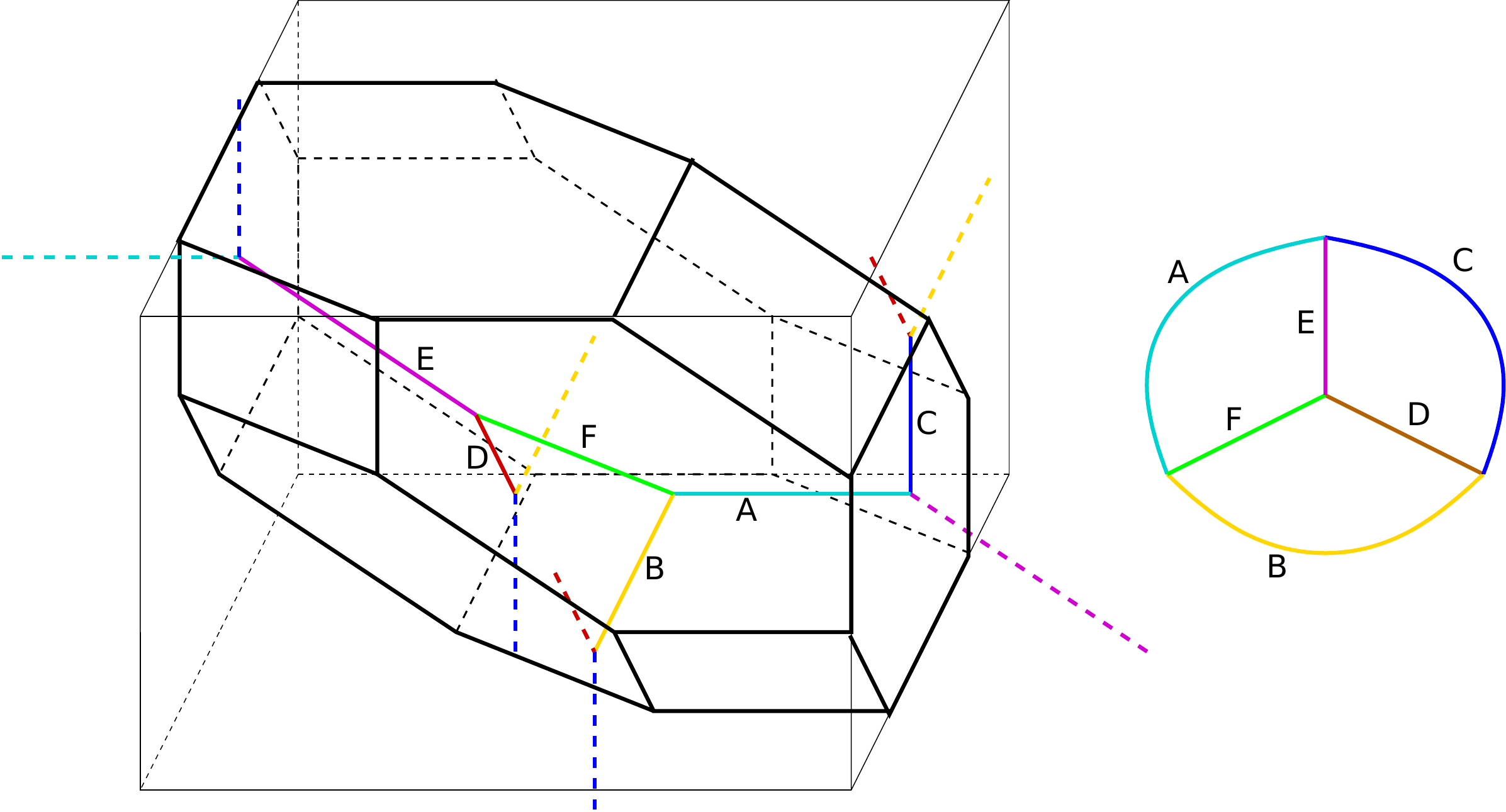} 
    \caption{A genus 3 curve and its image $W_1$ in the $J(C)$.}
     \label{fig:W1}
  \end{figure}
Fig. \ref{fig:W1} shows $W_1$ as a $\Gamma_1$-periodic 1-cycle in $V\cong \R^3$, the universal covering of $J(C)$. The space $V$ is filled by the maximal Voronoi cells, each is the zonotope, more precicsely,  the Minkowski sum of intervals corresponding to the edges of $C$. 
(cf.  Section \ref{section:zonotopes}). We will use same zonotope as a fundamental domain for future pictures as well.

We say two cycles $Z_1$ and $Z_2$ in $J$ are {\em algebraically equivalent}: $Z_1\sim Z_2$, if there is a tropical curve $S$, two points $s_1, s_2\in S$ and an algebraic cycle $W\subset J\times S$ (cf. \cite{Mik06}) such that
$$\pi_*[W\cap (J \times s_1) - W\cap (J \times s_2) ] = Z_1-Z_2 .
$$
Here $\cap$ means tropical (or stable) intersection (cf. \cite{Mik06}, \cite{Sturm}, \cite{Shaw}) and $\pi: J\times S \to J$ is the projection. 

Recall that the choice of the base point $p_0$ makes $J(C)$ into an abelian group. Notice, however, that another choice of $p_0$ will result in translation of the cycle $W_k$, which makes the choice of the base point irrelevant modulo algebraic equivalence. Given this group structure one can define an action by integers $n_*: J\to J$ by $n_*:  x \mapsto n x$, which descends to an action on the Chow group. The main purpose of this article is to compare $W_k$ with $W_k^-:= (-1)_* W_k$ modulo algebraic equivalence.

First we restrict our attention to the genus 3 connected curves. There are five combinatorial types of generic (trivalent) curves: $K_4$ and four hyperelliptic types (see Fig. \ref{fig:five_types}). 
 \begin{figure}[htb] 
     \centering
    \includegraphics[width=5.2in]{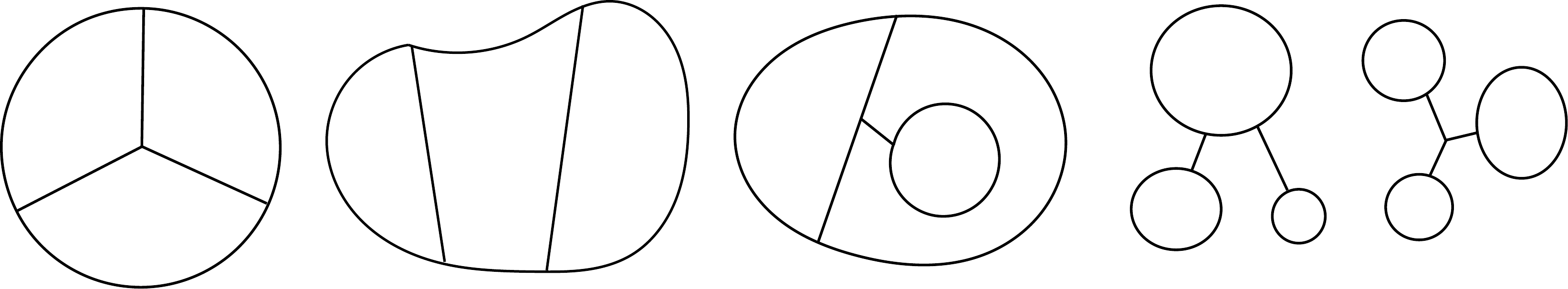}
         \caption{The five types of generic genus 3 curves.}
         \label{fig:five_types}
  \end{figure}
Other genus 3 curves are degenerations of these five.

\begin{theorem}\label{thm:main}
Let $C$ be a generic genus 3 curve of type $K_4$. Then  $W_1$ is not algebraically equivalent to $W_1^-:= (-1)_* W_1$ in $J(C)$.
\end{theorem}  

\begin{remark}
For curves of the hyperelliptic types we do have $W_1\sim W_1^-$. This reflects the failure of Torelli theorem for tropical curves. More precisely for a curve $C$ of a hyperelliptic type there is a hyperelliptic curve (that is the one which admits a degree 2 map to a rational curve) whose Jacobian coincides with $J(C)$. Moreover the $W_1$ for $C$ is a deformation of the Abel-Jacobi image of this hyperelliptic curve.
\end{remark}

The proof of the theorem mimics Ceresa's original proof \cite{ceresa} in the complex case with some simplifications. The connecting 3-chains are replaced by tropical chains with coefficients in $\Gamma_2$ and tropical (co)homology (cf. \cite{MZh13}) plays the r\^ole of the Hodge decomposition in $H^3(C)$. But we do not need to consider the full 6-dimensional family or use Griffiths transversality. Our argument works for any (generic) fixed tropical curve.

\subsection{Algebraic cycles and tropical homology}
Let $C$ be a curve of type $K_4$ and let $J$ be its Jacobian. Let $C_k(J, \Gamma_2)$ denote the group of polyhedral k-chains with coefficients in $\Gamma_2\cong \Z^3$. That is, an element $\gamma\in C_k(J, \Gamma_2)$ has the form $\sum _\sigma \beta_\sigma \sigma$, where $\beta_\sigma\in \Gamma_2$ (we call $\beta_\sigma$ the framing vector at $\sigma$) and $\sigma:\Delta \to J$ is a linear map from a polytope $\Delta \subset \R^k$ to $J$. Sometimes we will abuse notation and identify $\sigma$ with its image in $J$. The usual boundary map $\partial$ makes $C_\bullet (J, \Gamma_2)$ into a chain complex. The universal coefficient theorem allows us canonically identify homology $H_k(J,\Gamma_2)$ of this complex with the groups $\bigwedge^k \Gamma_1 \otimes \Gamma_2$.

Given an algebraic 1-cycle $Y$ in $J$ one can associate to it a {\em tautological} tropical cycle $[Y]\in C_1(J, \Gamma_2)$ as follows (cf. \cite{MZh13} for details). Every (oriented) edge $e\subset Y$ with weight $w_e$ defines a 1-cell $w_e \sigma_e$, and the primitive vector along this edge defines its framing $\beta_e$, an element in $\Gamma_2$. If the orientation of the edge is reversed, then so is the direction of the primitive vector. Both effects cancel in $ C_1(J, \Gamma_2)$.

Let $Z_1$ and $Z_2$ be two 1-cycles which are algebraically equivalent. Then their classes $[Z_1]=[Z_2]$ in $H_1(J,\Gamma_2)$, since one can view the algebraic equivalence as a deformation family connecting $Z_1$ and $Z_2$ (cf. \cite{Mik06}). Our next goal is to describe a particular element $\gamma\in  C_2(J, \Gamma_2)$ with $\partial \gamma = [Z_1] - [Z_2]$ on the chain level.

Suppose $W$, a cycle in $J\times S$, provides an algebraic equivalence between $Z_1$ and $Z_2$. Choose a path $P$ in $S$ between $s_1$ and $s_2$ and let $Z_s:=\pi_* (W\cap (J\times s))$ for any $s\in P$. Let $\pi':J\times S \to S$ denote the other projection and let $W_P$ be the restriction of $W$ to $(\pi')^{-1} (P)$. $W_P$ is a weighted polyhedral complex in $J\times S$ with weights inherited from $W$. 

Let $\tau$ be a 2-cell in the support of $W_P$ with weight $w_\tau$. If $\pi': J\times S \to S$ is transversal on $\tau$ we 
can define a vector field on $\tau$ by pulling back the tautological framing from the cycles $Z_s$, for $s\in \pi'(\tau)\subset S$. Since the maps $\pi$ and $\pi'$ are linear on $\tau$ this vector field is constant on $\tau$, and thus defines the coefficient $\beta_\tau$. If $\pi'$ is not transversal on $\tau$, that is it maps $\tau$ to a point in $S$, we set $\beta_\tau=0$.

\begin{lemma}\label{lemma:boundary}
Let $W$ be an algebraic equivalence between $Z_1$ and $Z_2$. Then the tropical 2-chain with coefficients defined above
$$\gamma= \sum_{\tau\in \supp (W_P)} w_\tau \beta_\tau \pi(\tau)
$$
connects the corresponding tropical cycles:  $\partial \gamma = [Z_1] - [Z_2]$.
\end{lemma}
\begin{proof}
By subdividing $S$ if necessary and using additivity of cycles and transitivity of algebraic equivalence it is enough to consider the case when $P$ is an edge and the map $\pi'$ is transversal on every $\tau\subset W_P$ in the preimage of the interior of $P$. Then the tropical (stable) intersection $W_P\cap (J\times s)$ used to define $Z_s$ becomes the usual set-theoretic intersection with some weights. 

Let $e$ be an interior edge of $W_P$. If $\pi'$ is transversal on $e$ then the weights from the stable intersection on the 2-cells adjacent to $e$ are determined by the relative position of $e$ with respect to $J\times s$ in $J \times S$ and hence all are the same. Then the balancing condition for the cycle $W$ along $e$ turns into the zero-boundary property of the chain $\gamma$. On the other hand if $\pi'(e)$ is a point in $S$ then we subdivide $P$ and use additivity as above.

Passing to the limit at the end points of $P$ and using continuity of the stable intersection we can identify  the boundary of $\gamma$ with the tautological cycle $[Z_1] - [Z_2]$. 
\end{proof}

\subsection{Determinantal 2-form and its periods.}
Let us fix a basis $e_1, e_2, e_3$ of $\Gamma_2$ and let $e^*_1, e^*_2, e^*_3$ be the dual basis of the dual lattice $\Gamma_2^*$. The choice of the basis defines canonically (up to translation) the linear coordinates $x_1, x_2, x_3$ on $J$. Let $dx_1, dx_2, dx_3$ be the corresponding constant 1-forms on $J$.

We will use notation $d \hat x_{i}=dx_{i+1}\wedge dx_{i+2}$, where we assume cyclic ordering of $\{1,2,3\}$. Then we define the constant 2-form with coefficients in $\Gamma_2^*$ as
$$\Omega_0:= \sum_{i=1}^3 e_i^* d \hat x_{ i}.
$$
In fact, the form is independent of the choice of the basis $\{e_i\}$ since with another choice it will change by the determinant of the basis transformation, which is 1.

We can integrate a $\Gamma_2^*$-valued 2-form along any tropical 2-chain in $\gamma\in  C_2(J, \Gamma_2)$ by first evaluating the coefficients point-wise and then taking the ordinary integral.

\begin{lemma}
Let $Z_1$ and $Z_2$ be two algebraically equivalent 1-cycles in $J$ and let $\gamma_0$ be a connecting chain between $[Z_1]$ and $[Z_2]$ as in Lemma \ref{lemma:boundary} above. Then $\int_{\gamma_0} \Omega_0= 0$.
\end{lemma}
\begin{proof}
For any triple of vectors $\beta, v_1,  v_2 \in \Gamma_2$ evaluating $\Omega_0$ on the expression $v_1\wedge v_2 \otimes \beta$ amounts to calculating the volume of the parallelepiped spanned by the triple. But note that on every $\sigma$  the corresponding framing vector $\beta_\sigma$ by construction of $\gamma_0$ lies in the linear span of $\<\sigma\>$. Thus $\Omega_0 |_{\beta_\sigma \sigma} \equiv 0$ for every $\sigma$ in the support of $\gamma_0$.
\end{proof}

We will need to calculate the periods of the form $\Omega_0$, that is the integrals over the elements in $H_2(J,\Gamma_2)=\bigwedge^2 \Gamma_1 \otimes \Gamma_2$. For this let us choose a basis $e_1, e_2, e_3$ of $\Gamma_2$ and a basis $\gamma_1, \gamma_2, \gamma_3$ of $\Gamma_1$ as shown on Fig. \ref{fig:K4_basis}. Here $e_1$ is the primitive tangent vector at any interior point of the edge A thought of as the evaluation functional on the space of 1-forms $\Omega(C)$. 
 \begin{figure}[htb]
    \centering
    \includegraphics[width=1.4in]{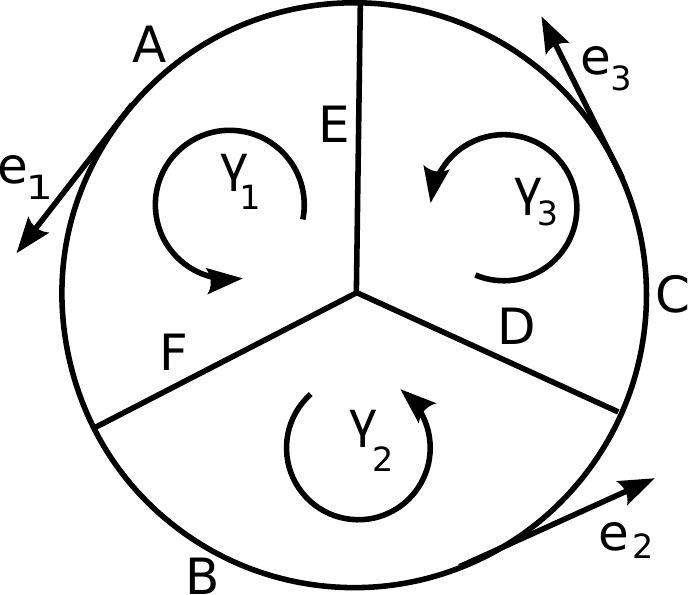}
    \caption{Bases for $\Gamma_2$ and $\Gamma_1$.} 
    \label{fig:K4_basis}
  \end{figure}
Similar $e_2$ and $e_3$ correspond to the edges B and C respectively.  Then we can write the three cycles $\gamma_1, \gamma_2, \gamma_3$ into the matrix 
$$Q=
\left(
\begin{smallmatrix}
a+e+f & -f & -e \\
-f & b+d+f & -d \\
-e & -d & c+d+e
\end{smallmatrix}
\right),
$$ 
whose columns are the coordinates of the $\gamma$'s in terms of the $\{e_i\}$.
Here $a, b, c, d, e, f  \in \R_{\ge 0}$ are the lengths of the corresponding edges of the curve.
 One can identify $Q$ as the matrix of the polarization form on the abelian variety $J$ in the basis $\{\gamma_1, \gamma_2, \gamma_3\}$.

Then the periods of $\Omega_0$ are generated over $\Z$ by the minors of $Q$:
$$\int_{\hat \gamma_i \otimes e_j} \Omega_0 = M_{ij}.
$$
Here we use the notation $\hat \gamma_i$ similar to $d \hat x_i$ (e.g., $\hat \gamma_1=\gamma_2 \wedge \gamma_3$). 
Explicitly, the 6 periods are:
\begin{gather*}
ab+ad+af+be+de+ef+bf+df,  \qquad ad+de+df+ef,\\
ac+ad+ae+ce+de+cf+df+ef, \qquad be+de+df+ef,\\
bc+bd+be+cd+de+cf+df+ef, \qquad cf+df+ef+de.
\end{gather*}
More symmetric set of the generators would be
\begin{gather*}
ad-be, \quad ad-cf,\\
de+df+ef+ad,\\
ab+af+bf+ad,\\
ac+ae+ce+ad,\\
bc+bd+cd+ad.
\end{gather*}
In other words the periods are generated by the differences of the products of opposite pairs of edges and sums of products of all three pairs adjacent to a common vertex plus an opposite pair. 

\subsection{Proof of Theorem \ref{thm:main}}
Now we are ready to finish the proof of the theorem.

\begin{proof}[Proof of Theorem \ref{thm:main}]
We will choose a convenient representative of the cycle $W_1-W_1^-$ and a 2-chain $\gamma_0 \in C_2(J,\Gamma_2)$ with $\partial \gamma_0 = [W_1] - [W_1^-]$. Another choice $\gamma_0'$ of the connecting chain will differ from $\gamma_0$ by a cycle in  $C_2(J, \Gamma_2)$. Thus to prove the theorem it is sufficient to show that $\int_{\gamma_0} \Omega_0\ne 0$ modulo periods of $\Omega_0$. 

First we move $W_1$ and $W_1^-$ such that the image of one of the edges, say C (of blue color) in $W_1$ coincides with that of $W_1^-$, thus canceling in the  $W_1-W_1^-$. For $\gamma_0$ we choose the 2-chain supported on 5 parallelograms: MNPQ, NLKP, LMQK, KQSP and PSML (see Fig. \ref{fig:2-chain}). 
 \begin{figure}[htb] 
     \centering
    \includegraphics[width=3.5in]{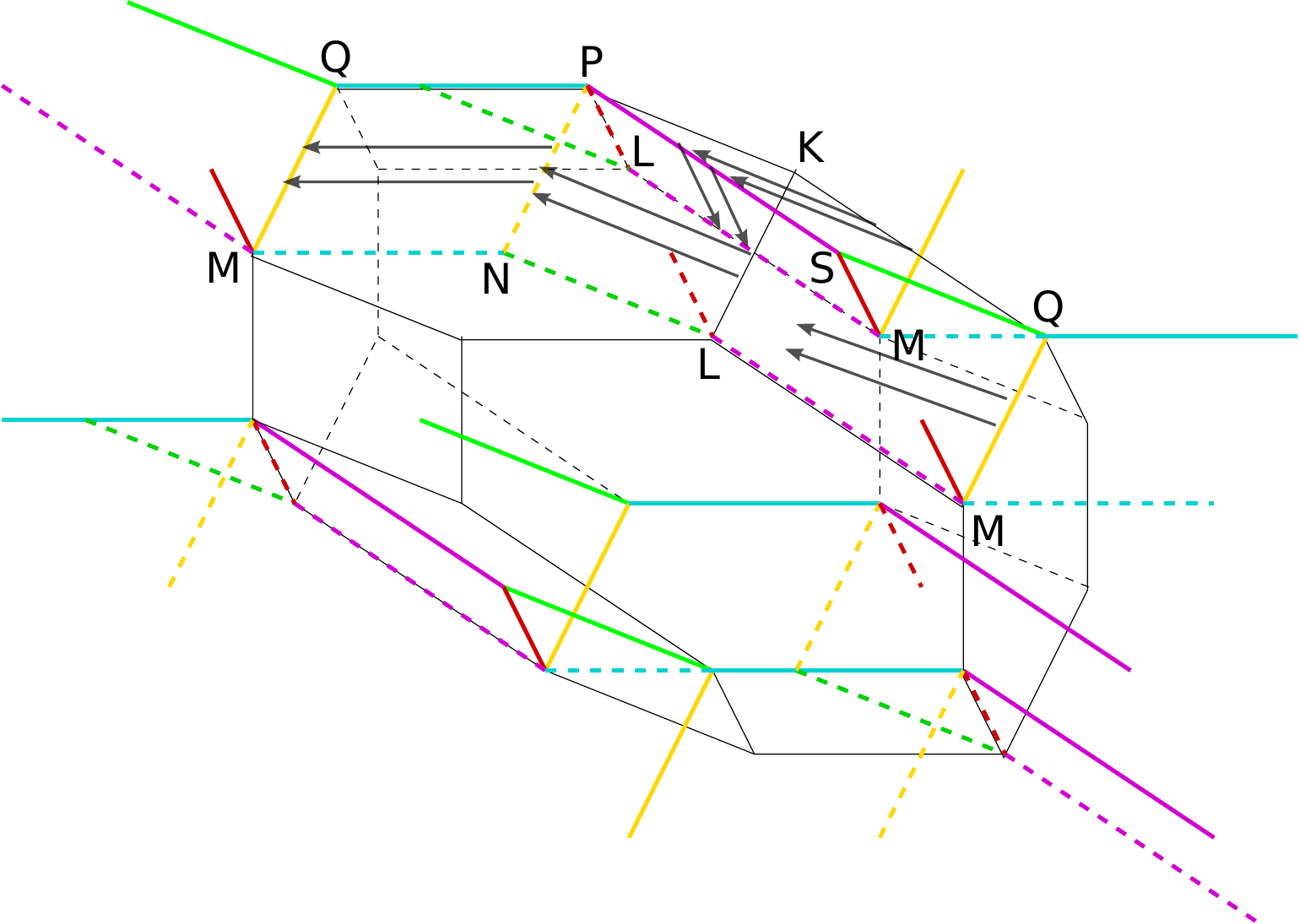}
       \caption{$W_1$ (solid colors) and  $W_1^-$ (dashed colors) and the connecting 2-chain.}
          \label{fig:2-chain}
  \end{figure}
The orientation of all cells are counterclockwise if viewed from above (there are no vertical cells in $\gamma_0$). 

To help a reader visualize the picture 
 \begin{figure}[htb] 
     \centering
    \includegraphics[width=2.5in]{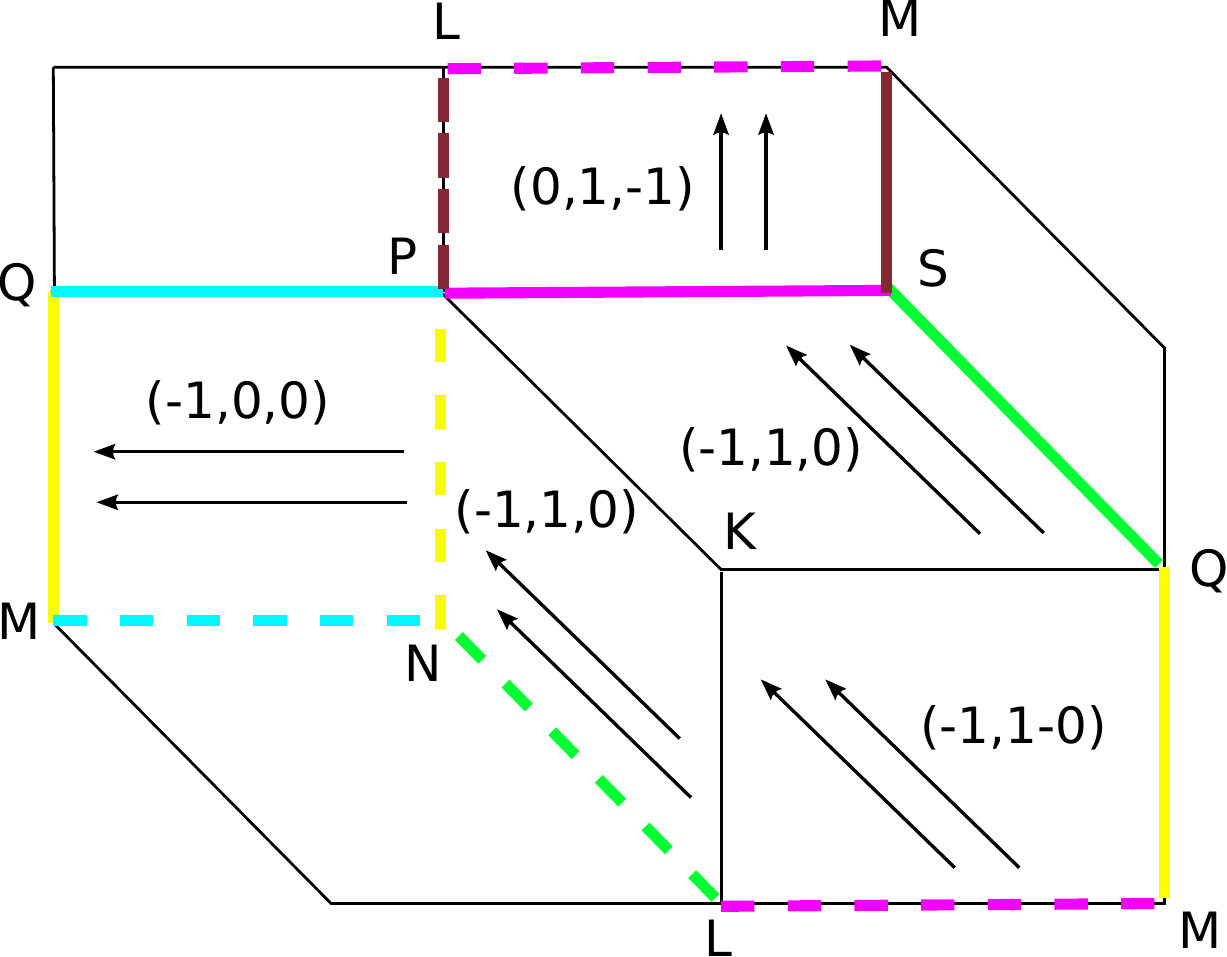}
      \caption{A view from above $\partial \gamma_0=[W_1] - [W_1^-]$.}
         \label{fig:2-chain_up}
  \end{figure} 
we give a view of the chain from above, where we indicated the coordinates of the framing vectors on all five parallelograms (see Fig. \ref{fig:2-chain_up}).
It is an easy check that $\partial \gamma_0=[W_1] - [W_1^-]$.
We also see that the framing on all parallelograms except LMQK is parallel to the supporting cell. Thus these four parallelograms contribute 0 into the integral. The integration along LMQK contributes minus its area $-ad$ into $\int_{\gamma_0}\Omega_0$.

But the real number $ad$ is not in the lattice of the periods for generic choice of the edge lengths $a, b, c, d, e, f  \in \R$. This completes the proof.
\end{proof}

\begin{remark}
Note that if one of the edges collapses, then $\int_{\gamma_0}\Omega_0 =- ad$ become zero modulo periods. This shows that the proof does not work for the hyperelliptic type curves. In fact one can easily construct an explicit deformation between $W_1$ and $W_1^-$ in $J$ in this case (see Fig. \ref{fig:hyperelliptic} for a curve of the second type from Fig. \ref{fig:five_types}).
 \begin{figure}[htb] 
     \centering
    \includegraphics[width=3.5in]{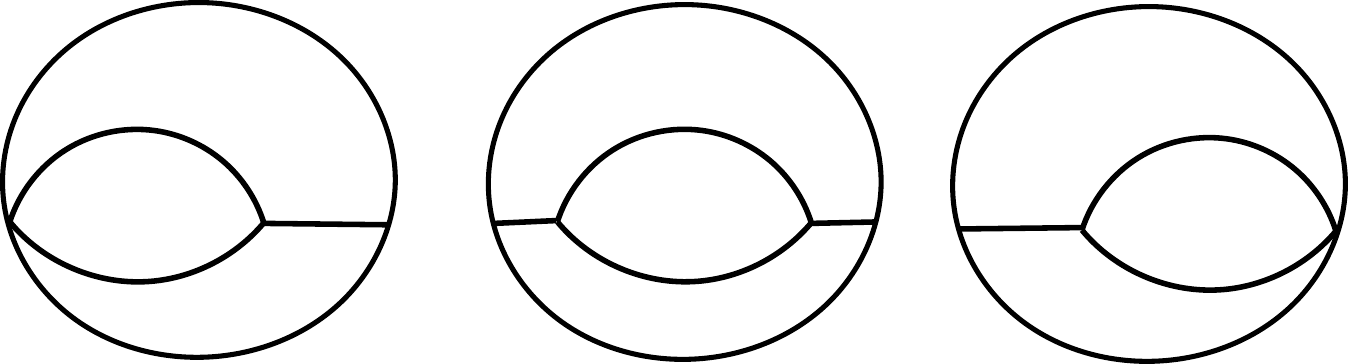}
      \caption{Deformation from $C$ to $C^-$ with hyperelliptic curve in the middle (all curves in the family have the same Jacobian).}
         \label{fig:hyperelliptic}
  \end{figure}  
\end{remark}

\section{Dicings, zonotopes and higher genus curves}\label{section:zonotopes} 
By {\em lattice} we mean a positive definite quadratic form $Q$ on a free abelian group $\Lambda$. In our case the group $\Lambda$ will be $\Gamma_2^* \cong \Z^g$. Since $Q$ provides an isomorphism $\Gamma_2^*\cong \Gamma_1$ this is equivalent to define $Q$ on $\Gamma_1$ as in the definition of the tropical abelian varieties.

We say a lattice $(\Lambda,Q)$ is a {\em dicing} if the associated Delaunay decomposition of the real vector space $\Lambda\otimes \R \cong \R^g$ is given by families of parallel hyperplanes intersecting at the lattice points. There are two other equivalent formulations of the dicing condition (cf., e.g. \cite{erdahl}). 
First is that the defining quadratic form can be written as $Q=\sum \alpha_i (e_i)^2$, where the collection of linear functionals $\{e_i\}$ form a totally unimodular system in the dual lattice $\Lambda^*$.  Second is that the maximal Voronoi cell is the zonotope, with zone vectors $\alpha_i e_i$.

Let us now return to the case of tropical Jacobian. To any edge one can associate an element $e_i$ of $\Gamma_2$ (as we did in the proof of Theorem \ref{thm:main}) as the evaluation of 1-forms on a primitive vector at any interior point of the edge. Then we can write the polarization form as $Q=\sum \alpha_i (e_i)^2$ (the choice of signs for the $e_i$'s won't matter).  Here $\alpha_i$ is the length of the corresponding edge. Thus we have the following:

\begin{proposition}\label{prop:zonotope}
The Jacobian lattice $(\Gamma_2^*, Q)$ is a dicing with zone vectors given by the edges of $C$.
\end{proposition}

It is convenient to take the Voronoi zonotope as the fundamental domain of  $\Gamma_1$ action on $V=\Omega^*$. That is what we do on all pictures in the paper. Let us denote by $\mathcal Z(C)$ the zonotope arising  from a curve $C$. 

\begin{proposition}\label{prop:projection}
Let $C_e$ be a curve obtained from $C$ by removing an edge $e$.
Then $\mathcal Z(C_{e})$ is a projection of $\mathcal Z(C)$ along the zone vector corresponding to $e$  (see Fig. \ref{fig:zono_projection}).
 \begin{figure}[htb]
    \centering
    \includegraphics[width=3in]{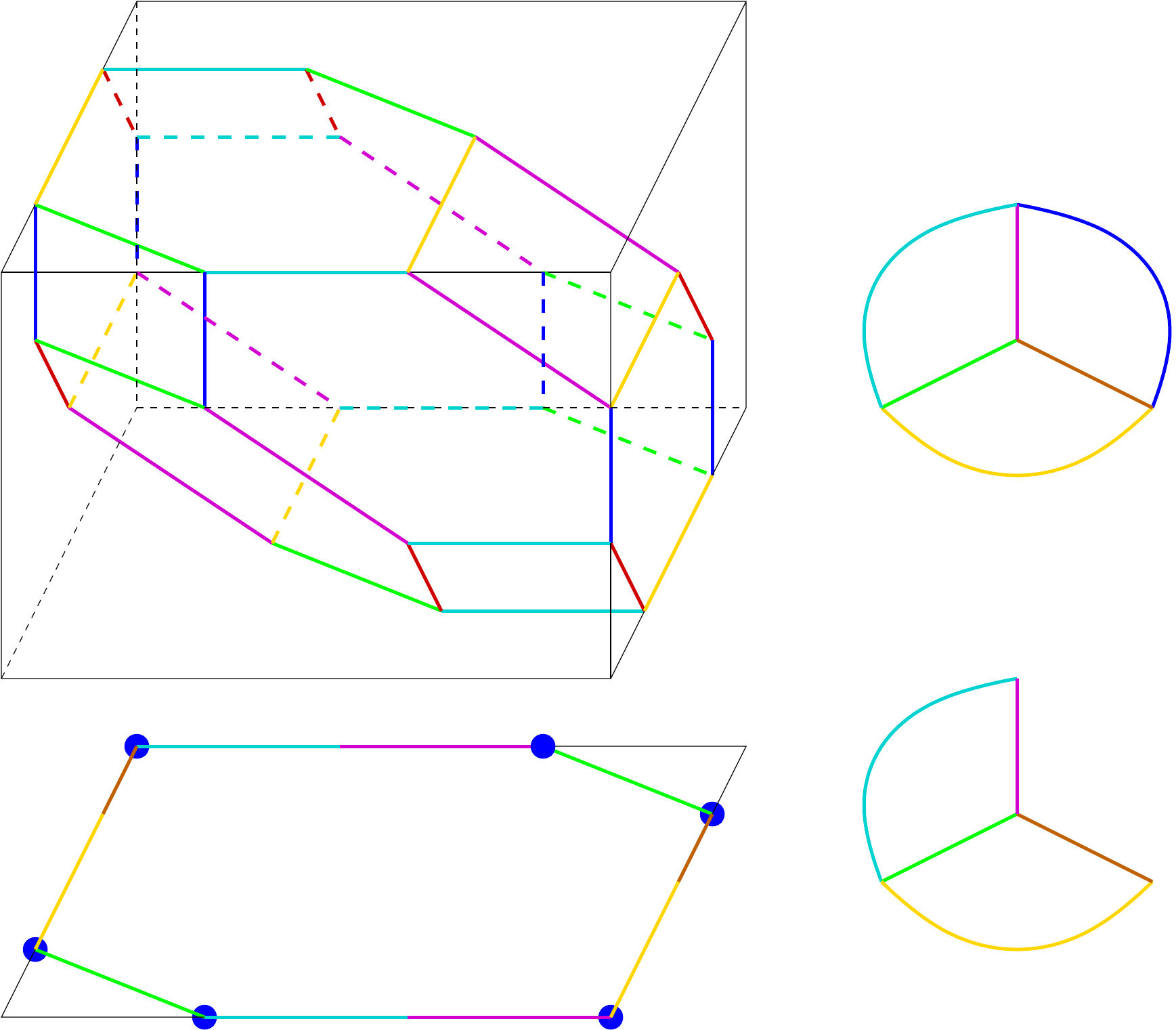}
    \caption{Removing an edge reduces genus and projects the zonotope.}
      \label{fig:zono_projection}
  \end{figure}
\end{proposition}

\begin{proof}
Removing an edge effectively corresponds to considering only those 1-forms (currents) on $C$ which do not pass through the missing edge. Thus on the dual side it results in considering linear functionals modulo the one which evaluates on this edge.  
\end{proof}

\begin{remark} 
It may be worth to note the effect of contracting some edges in $C$. Namely, let $C_{C_k}$ be the curve obtained from $C$ by contracting a genus $k$ subcurve $C_k\subset C$. This can be thought of as setting the lengths of all edges in $C_k$ to zero. Then one can argue that
$\mathcal Z(C_{C_k})$ is a face of of codimension $k$ in $\mathcal Z(C)$ (see Fig. \ref{fig:faces}). 
 \begin{figure}
    \centering
    \includegraphics[width=3in]{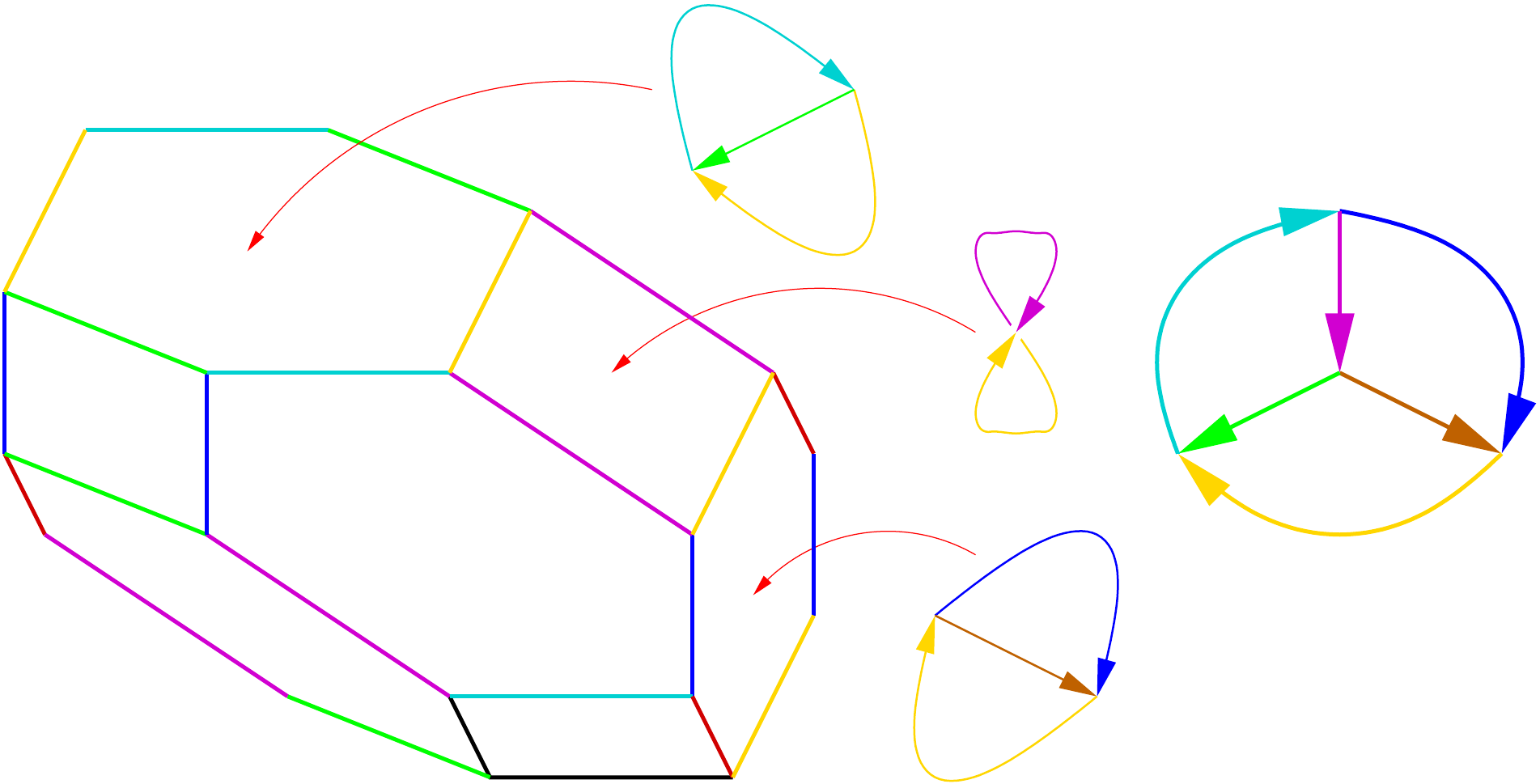}
      \caption{Contraction of subcurves in $C$ and faces of $\mathcal Z(C)$.}
          \label{fig:faces}
  \end{figure}
Moreover there is a bijection between faces of $\mathcal Z(C)$ and subcurves of $C$.
\end{remark}

Proposition \ref{prop:projection} combined with Theorem \ref{thm:main} leads to the main result of the paper:
\begin{theorem}
Let $C$ be a tropical curve of genus $g$ whose underlying graph contains $K_4$ as a subgraph. Then $W_k$ is not algebraically equivalent to $W_k^-$ in $J(C)$ for $k=1,\dots, g-2$.
\end{theorem}
\begin{proof}
Suppose $W_1\sim W_1^-$. This equivalence survives in the projection $J(C) \to J(C_e)$ (which is a tropical map), where $C_e$ is obtained from $C$ by removing an edge.  Finally we arrive at a genus 3 curve of type $K_4$ which provides the contradiction.

As for $W_k$ with $k>1$ we can formally follow Ceresa's inductive argument in \cite{ceresa}. Namely we can degenerate $C$ to a curve $\tilde C$ which is a union of a (generic) curve $C'$ of genus $g-1$ and  a loop $E$ (elliptic curve) with a common vertex. Then $W_k$ decomposes as a product $W_{k-1}(C')\times E$ in $J(C')\times E = J(\tilde C)$ and we use the induction hypothesis.
\end{proof}


\begin{thebibliography}{KKMS73}

\bibitem[Ale04]{alexeev}
V. Alexeev. 
\newblock Compactified Jacobians and Torelli map. Publ. RIMS, Kyoto Univ. 40 (2004),
\newblock {\em Publ. RIMS, Kyoto Univ.} 40 (2004), 1241--1265.


\bibitem[Cer83]{ceresa}
G. Ceresa.
\newblock $C$ is not algebraically equivalent to $C^-$ in its Jacobian.
\newblock {\em Ann. of Math.} 117 (1983), 285 -- 291.

\bibitem[Erd99]{erdahl}
R. Erdahl. 
\newblock Zonotopes, dicings and Voronoi's conjecture on parallelohedra.
\newblock {\em European J. Combin.} 20 (1999), no. 6, 527--549.

\bibitem[Mik06]{Mik06}
G. Mikhalkin.
\newblock Tropical geometry and its application. 
\newblock {\em Proceedings of the International Congress of Mathematicians}, Madrid 2006, 827--852.

\bibitem[MZh07]{MZ2}
G. Mikhalkin and I. Zharkov.
\newblock Tropical curves, their {J}acobians and {T}heta functions.
\newblock  In {\em Curves and Abelian Varieties} (V. Alexeev, A. Beauville, H. Clemens and E. Izadi (Eds.)), Contemp. Math., Vol. 465, AMS 2008, 203--230.

\bibitem[MZh13]{MZh13}
G. Mikhalkin and I. Zharkov. 
\newblock Tropical eigenwave and intermediate Jacobians.
\newblock  arXiv:1302.0252.

\bibitem[Sh10]{Shaw}
K. Shaw.
\newblock A tropical intersection product in matroidal fans. 
\newblock arXiv:1010.3967.

\bibitem[RST05]{Sturm}
J. Richter-Gebert, B. Sturmfels, T. Theobald. 
\newblock First steps in tropical geometry. 
\newblock In {\em Idempotent mathematics and mathematical physics}, Contemp. Math., 377, Amer. Math. Soc., Providence, RI, 2005, 289--317. 

\bibitem[Viv12]{viviani}
F. Viviani.
\newblock Tropicalizing vs compactifying the Torelli morphism.
\newblock  arXiv:1204.3875.

\end{thebibliography}
\end{document}